\documentclass[10pt, letterpaper]{article}
\usepackage{fullpage}
\usepackage{amsmath, amssymb}

\newtheorem{theorem}{Theorem}[section]
\newtheorem{lemma}[theorem]{Lemma}
\newtheorem{proposition}[theorem]{Proposition}
\newtheorem{corollary}[theorem]{Corollary}
\newtheorem{conjecture}[theorem]{Conjecture}
\numberwithin{equation}{section}
 
\newenvironment{proof}[1][Proof]{\begin{trivlist}
\item[\hskip \labelsep {\bfseries #1}]}{\end{trivlist}}
\newenvironment{definition}[1][Definition]{\begin{trivlist}
\item[\hskip \labelsep {\bfseries #1}]}{\end{trivlist}}

\newenvironment{remark}[1][Remark]{\begin{trivlist}
\item[\hskip \labelsep {\bfseries #1}]}{\end{trivlist}}

\def\Z{{\mathbb Z}}         
\def\N{{\mathbb N}}      
\def\R{{\mathbb R}}       
\def\Q{{\mathbb Q}}       
\def\T{{\mathbb T}}        
\def\F{{\cal F}}                 
\def\B{{\mathcal B}}               
\def\S{{\mathcal S}}               
\def\supp{\text{supp }}

\title{Pointwise Convergence for Subsequences of Weighted Averages}
\author{Patrick LaVictoire\footnote{Supported in part by NSF Grant DMS-0401260.}, UW Madison}

\begin{document}
\maketitle

\begin{abstract}
We prove that if $\mu_n$ are probability measures on $\Z$ such that $\hat \mu_n$ converges to 0 uniformly on every compact subset of $(0,1)$, then there exists a subsequence $\{n_k\}$ such that the weighted ergodic averages corresponding to $\mu_{n_k}$ satisfy a pointwise ergodic theorem in $L^1$.  We further discuss the relationship between Fourier decay and pointwise ergodic theorems for subsequences, considering in particular the averages along $n^2+ \lfloor \rho(n)\rfloor$ for a slowly growing function $\rho$.  Under some monotonicity assumptions, the rate of growth of $\rho'(x)$ determines the existence of a ``good'' subsequence of these averages.
\end{abstract}

\section{Introduction}
Generally speaking, if we have a family of operators $T_n$ on a Banach space $V$ which converge in some weak sense, we might ask whether there exists a subsequence $T_{n_k}$ which converges in some stronger sense.  An important special case here is the contrast between various types of ``convergence in the mean'' and ``convergence almost everywhere'', as for example in the following recent result of Kostyukovsky and Olevskii \cite{MR2177432} on approximate identities:
\begin{definition}
A sequence of functions $\phi_n\in L^1(\R)$ is an \emph{approximate identity on $\R$} if $\|\phi_n\ast f-f\|_1\to0$ as $n\to\infty$, for all $f\in L^1(\R)$.
\end{definition}
\begin{theorem}
Let $\{\phi_n\}$ an approximate identity on $\R$ consisting of nonnegative functions.  Then there is a sequence $\{n_k\}$ such that $\phi_{n_k}\ast f\to f$ a.e. for every $f\in L^1(\R)$.
\end{theorem}
As noted by Rosenblatt \cite{MR2497325}, this example is analogous to an open question about the pointwise convergence of subsequences of certain weighted ergodic averages.  In that context, the natural analogue to an approximate identity is a sequence of probability measures $\{\mu_n\}$ such that for any ergodic dynamical system $(X,\F,m,\tau)$ and any $f\in L^1$, the weighted averages
\begin{eqnarray}
\mu_{n_k}f(x):=\sum_{j\in\Z} f(\tau^j x)\mu_{n_k}(j)
\end{eqnarray}
converge in the $L^1$ norm to $\int_X f dm$.  This is equivalent (\cite{BJR}, Proposition 1.7b and Corollary 1.8) to the Fourier condition $\hat\mu_n(\gamma)\to0\;\forall \gamma\in(0,1)$.  However, a stronger condition seems to be required for an analogous result:
\begin{definition}
We say that a sequence $\{\mu_n\}$ of probability measures on $\Z$ has \emph{asymptotically trivial transforms} if $\hat \mu_n$ converges to 0 uniformly on every compact subset of $(0,1)$, or equivalently, if $$\sup_{\gamma\in[0,1)}|(1-e(\gamma))\hat \mu_n(\gamma)|\to0,$$
where we denote $e(\gamma)=e^{2\pi i\gamma}$.
\end{definition}
Bellow, Jones and Rosenblatt \cite{BJR} proved the following:
\begin{theorem}
\label{lpresult}
Suppose $\{\mu_n\}$ is a sequence of probability measures on $\Z$ with asymptotically trivial transforms.  Then there exists a subsequence $\{n_k\}$ such that $\mu_{n_k}f(x)$ converges a.e. for every dynamical system $(X,\F,m,\tau)$ and every $f\in L^p$, $p>1$.
\end{theorem}
This was proved by an analysis of square functions and by interpolation from $L^2$ to $L^p$, which left open the $L^1$ question (see Section 4 in \cite{BJR}).  In this paper, I prove the following weak-type (1,1) maximal inequality on $\Z$:
\begin{theorem}
\label{weakmax}
Suppose $\{\mu_n\}$ has asymptotically trivial transforms.  Then there is a subsequence $\{n_k\}$ which obeys the weak type maximal inequality 
\begin{eqnarray*}
|\{x:\sup_k |\varphi\ast\mu_{n_k}(x)|>\lambda\}|\leq C\lambda^{-1}\|\varphi\|_{\ell^1(\Z)}\;\forall\varphi\in\ell^1(\Z).
\end{eqnarray*}
\end{theorem}
Given Theorem \ref{lpresult} and the Conze principle \cite{Conze}, this implies the full $L^1$ result:
\begin{corollary}
\label{l1result}
Suppose $\{\mu_n\}$ has asymptotically trivial transforms.  Then there exists a subsequence $\{n_k\}$ such that $\mu_{n_k}f(x)$ converges a.e. for every dynamical system $(X,\F,m,\tau)$ and every $f\in L^1(X)$.
\end{corollary}
The next question is whether our stronger hypothesis (that $\{\mu_n\}$ has asymptotically trivial transforms) can be replaced by the weaker one (that $\hat\mu_n(\gamma)\to0$ for all $\gamma\in(0,1)$).  If we remove the requirement that the $\mu_n$ are positive measures (but still require $\|\mu_n\|_1\leq1$), we can provide an explicit counterexample:
\begin{proposition}
\label{rotsqrs}
Consider the measures $\mu_n:=\displaystyle\frac1n\sum_{j=1}^n \delta_{j^2}e(n^{-1/2}j)$.  Then $\hat\mu_n(\gamma)\to0$ for all $\gamma\in[0,1)$, but for any subsequence $\{n_k\}$ and any (non-atomic) ergodic dynamical system $(X,\F,m,\tau)$, there exists an $f\in L^1(X)$ such that $\mu_{n_k}f(x)$ diverges on a set of positive measure in $X$.
\end{proposition}
\begin{remark}
Note that the proof of a.e. convergence for $L^2$ functions (Theorem 1.14 in \cite{BJR}) does not require positivity of the measures; if we trace it through, it requires only the assumption that the $\mu_n$ have asymptotically trivial transforms, bounded $\ell^1$ norm, and that $\lim_{n\to\infty} \hat\mu_n(1)$ exists.
\end{remark}
With the hypothesis of positivity restored, the above question remains open.  However, we strongly suspect that some type of uniformity is essential, and that some positive variant on the above example can be made to work.
\begin{conjecture}
There exists a sequence of probability measures $\{\mu_n\}$ such that $\hat\mu_n(\gamma)\to0$ for all $\gamma\in(0,1)$, but for any subsequence $\{n_k\}$ and any (non-atomic) ergodic dynamical system $(X,\F,m,\tau)$, there exists an $f\in L^1(X)$ such that $\mu_{n_k}f(x)$ diverges on a set of positive measure in $X$.
\end{conjecture}
Finally, we examine a special case: the averages along the sequence $a_k:=k^2+\lfloor \rho(k)\rfloor$, where $\rho$ is slowly growing.  Here, a positive result obtained via Corollary \ref{l1result} and a negative result obtained by the methods of \cite{LaVic2} meet at an exact threshold:
\begin{theorem}
\label{threshold}
Let $\rho\in C^2[0,\infty)$, with $\rho(x)\nearrow\infty$, $\rho'(x)\searrow0$ and $\rho''(x)\nearrow0$ as $x\to\infty$, such that for some $\epsilon>0$, $\rho'(x)\lesssim x^{-(\epsilon+2/3)}$ as $x\to\infty$. Consider the sequence of measures
$$\mu_N:=\frac1N\sum_{k=1}^N \delta_{k^2+\lfloor \rho(k)\rfloor}.$$
If $\rho'(x)\gg x^{-1}$ (thus $\rho(x)\gg \log x$), then the $\{\mu_N\}$ have asymptotically trivial transforms, and thus there exists a subsequence $\mu_{N_k}$ such that $\mu_{n_k}f(x)$ converges a.e. for every dynamical system $(X,\F,m,\tau)$ and every $f\in L^1(X)$.\\ \\
If $\rho'(x)\lesssim x^{-1}$ (thus $\rho(x)\lesssim \log x$), then for any subsequence $\{n_k\}$ and any (non-atomic) ergodic dynamical system $(X,\F,m,\tau)$, there exists an $f\in L^1(X)$ such that $\mu_{n_k}f(x)$ diverges on a set of positive measure in $X$.
\end{theorem}
\begin{remark}
This requires an additional monotonicity assumption (the existence of $\lim_{x\to\infty}x\rho'(x)\in[0,\infty]$) in order to become a true dichotomy. Compare \cite{EAS}, where a stricter condition on derivatives (namely, that $\rho$ belongs to a Hardy field) allows Boshernitzan et al. to establish pointwise ergodic theorems in $L^2$ based solely on the rate of growth of such perturbations.  
\end{remark}
\begin{remark}
The condition $\rho'(x)\lesssim x^{-(\epsilon+2/3)}$ is an artifact of the proof rather than a genuine restriction. The requirement that $\rho''\nearrow0$, though, is necessary in some form to establish asymptotically trivial transforms; it is simple otherwise to create examples such that the exponential sums in (\ref{expos}) do not settle down away from 0.
\end{remark}

\section{Positive Result: Asymptotically Trivial Transforms}
The proof of Theorem \ref{weakmax} makes use of the following technique: given the Calder\'on-Zygmund decomposition of a function $f=g+\sum_s b_s$, we classify $s$ as ``small", ``large" or ``intermediate" with respect to each term of our subsequence $\mu_n$.  For $s$ ``large", we can use a covering lemma to handle the terms; for $s$ ``small", we will use cancellation properties of $b_s$; and since for each $s$ there will be only one $n$ for which it counts as ``intermediate", we can handle these terms with a trivial $L^1$ estimate.  This idea plays a role in \cite{STW} as well as other papers.
\\ \\
The proof will also use a technique in singular integral theory, developed by Fefferman \cite{CF} and Christ \cite{MC} and first applied to ergodic theory by Urban and Zienkiewicz \cite{UZ}, which uses a sufficiently powerful $L^2$ estimate to prove a weak $L^1$ estimate.
\\ \\
We may assume that $\|\mu_n\|_{\ell^1(\Z)}\leq1\;\forall n$, and write $\mu_n=\mu'_n+\eta_n$, where $\mu'_n$ is compactly supported and $\sum_{n=1}^\infty \|\eta_n\|_{\ell^1(\Z)}<\infty$.  Then
\begin{eqnarray*}
|\{x:\sup_n |\varphi\ast\eta_n(x)|>\lambda\}|&\leq& \lambda^{-1}\sum_{n=1}^\infty \|\varphi\ast\eta_n\|_{\ell^1(\Z)}\\
&\leq&\lambda^{-1}\left(\sum_{n=1}^\infty \|\eta_n\|_{\ell^1(\Z)}\right)\|\varphi\|_{\ell^1(\Z)}.
\end{eqnarray*}
Now $|\mu'_n(\gamma)|\leq |\hat\mu_n(\gamma)|+2\|\eta_n\|_1$, and so $\hat\mu'_n$ converges to 0 uniformly on every compact subset of $(0,1)$.  Thus we may assume that $\mu_n$ is compactly supported for each $n$.
\\ \\
Furthermore, if the union of these supports were compact, then it is easy to see (by Parseval's Theorem) that $\|\mu_n\|_{\ell^1(\Z)}\to0$ and we may choose a subsequence such that $\sum_k\|\mu_{n_k}\|_{\ell^1(\Z)}<\infty$; such a subsequence would trivially satisfy a weak maximal inequality.\\ \\
We may therefore assume that the union of the supports of the $\mu_n$ is unbounded, and accordingly set $S(n):=\min\{s\geq0: \text{ supp }\mu_m\subset[-2^s,2^s]\;\forall m\leq n\}$, and $N(s):=\min\{n: S(n)>s\}$.
\\
\\ Since we will want the cancellation properties of $\mu_{n+1}$ to overcome the size of the support of $\mu_n$, we choose an increasing subsequence $\{n_k\}$ such that $$\sup_{\gamma\in[0,1)}|(1-e(\gamma))\hat\mu_{n_{k}}(\gamma)|\leq 2^{-2S(n_{k-1})-2k}$$ and such that $S(n_k)$ is strictly increasing.  By passing to this subsequence, we may without loss of generality assume that $\mu_n$ has the following properties in the first place:
\begin{eqnarray}
&\text{supp }\mu_n \subset [-2^{S(n)},2^{S(n)}]& \\
\label{decaycdn}
&\sup_{\gamma\in[0,1)}|(1-e(\gamma))\hat\mu_n(\gamma)| \leq 2^{-2S(n-1)-2n}.&
\end{eqnarray}
Now, given $\varphi\in \ell^1$ and $\lambda>0$, we perform the discrete Calder\'on-Zygmund decomposition: we obtain a collection $\B$ of dyadic discrete intervals $Q_{s,k}$, and a decomposition $\varphi=g+\sum_{(s,k)\in\B}b_{s,k}$ with $\|g\|_\infty\leq \lambda$, such that for all $(s,k)\in\B$,
\begin{eqnarray*}
\supp b_{s,k} &\subset& Q_{s,k}\\
\sum_x b_{s,k}(x)&=&0\\
\sum_x |b_{s,k}(x)|&\leq& \lambda |Q_{s,k}|=2^s\lambda
\end{eqnarray*}
and such that
\begin{eqnarray*}
\sum_{(s,k)\in\B} |Q_{s,k}|&\leq& \lambda^{-1}\|\varphi\|_1.
\end{eqnarray*}
Let $b_s:=\sum_k b_{s,k}$ for each $s$, and let $Q_{s,k}^\star$ denote the interval with the same center as $Q_{s,k}$ and 3 times the length.  Then
\begin{eqnarray*}
|\{x:\sup_n |\mu_n\ast \varphi(x)|>3\lambda\}|&\leq &|\{x: \sup_n |\mu_n\ast g(x)|>\lambda\}|+|\{x:\sup_n |\mu_n\ast b(x)|>2\lambda\}|\\
&\leq& 0+\sum_{s,k}|Q_{s,k}^\star|+|\{x\not\in \bigcup_{s,k} Q_{s,k}^\star :\sup_n |\mu_n\ast b(x)|>2\lambda\}|\\
&\leq&  \frac C\lambda\|\varphi\|_1+|\{x: \sup_n |\mu_n\ast \sum_{s< S(n)}b_s(x)|>2\lambda\}|,
\end{eqnarray*}
because $\|\mu_n\ast g\|_\infty\leq \|\mu_n\|_1\|g\|_\infty\leq\lambda$ and because $s\geq S(n)\implies \text{ supp } \mu_n\ast b_{s,k} \subset Q_{s,k}^\star$.
\\ \\Now $S(n-1)\leq s <S(n)\implies n=N(s)$, and therefore we decompose
\begin{eqnarray*}
\sup_n |\mu_n\ast \sum_{s< S(n)}b_s(x)| &\leq & \sum_n |\mu_n\ast \sum_{s=S(n-1)}^{S(n)-1}b_s(x)| + \sup_n |\mu_n\ast \sum_{s< S(n-1)}b_s(x)| \\
&\leq & \sum_s |\mu_{N(s)}\ast b_s(x)| + \sup_n |\mu_n\ast \sum_{s< S(n-1)}b_s(x)|.
\end{eqnarray*}
As mentioned earlier, we can trivially bound the contribution from the ``intermediate'' terms:
\begin{eqnarray*}
|\{x: \sum_s |\mu_{N(s)}\ast b_s(x)|>\lambda\}|&\leq& \lambda^{-1} \sum_s \|b_s\ast \mu_{N(s)}\|_1\\
&\leq& \lambda^{-1} \sum_s \|b_s\|_1 \|\mu_{N(s)}\|_1\leq \frac C\lambda\|\varphi\|_1.
\end{eqnarray*}
We have thus reduced this problem to the following claim:

\begin{lemma}
\begin{eqnarray}
\label{lastset}
|\{x: \sup_n |\mu_n\ast \sum_{s< S(n-1)}b_s(x)|>\lambda\}|\leq\frac C\lambda\|\varphi\|_1.
\end{eqnarray}
\end{lemma}
\begin{proof}
We will be able to use (\ref{decaycdn}) to our advantage here, since each $b_{s,k}$ has mean 0 when averaged over dyadic intervals of size $2^{S(n-1)}$, and the Fourier bounds on $\mu_n$ are strong enough to exploit this.
\\ \\We consider the standard $\ell^1$ averages 
\begin{eqnarray}
\sigma_n=2^{-S(n-1)-n} \chi_{[1,2^{S(n-1)+n}]},
\end{eqnarray}
and decompose $\mu_n =\mu_n\ast\sigma_n +\mu_n\ast(\delta_0-\sigma_n )$.  Accordingly, the set on the left of (\ref{lastset}) is contained in the union of the sets
\begin{eqnarray*}
E_1&:=&\{x: \sup_n |(\mu_n\ast\sigma_n)\ast \sum_{s< S(n-1)}b_s(x)|>\frac\lambda2\},\\
E_2&:=&\{x: \sup_n |(\mu_n-\mu_n\ast\sigma_n)\ast \sum_{s< S(n-1)}b_s(x)|>\frac\lambda2\}.
\end{eqnarray*}
Observe that for any $t>s$,
\begin{eqnarray*}
|\chi_{[1,2^t]}\ast b_{s,k}(x)|\leq\left\{\begin{array}{ll} 0, & x\not\in Q_{s,k}+[0,2^t] \\
0,& x\in Q_{s,k}+[2^s,2^t-2^s]\\
\|b_{s,k}\|_1 & \text{otherwise,}\end{array}\right.
\end{eqnarray*}
since each $b_{s,k}$ has mean 0 and is supported on $Q_{s,k}$.  Therefore $\|\sigma_n\ast b_{s,k}\|_1\leq 2^{-S(n-1)-n+s+1}\|b_{s,k}\|_1$, which implies
\begin{eqnarray*}
|E_1| &\leq & 2\lambda^{-1}\sum_n \|\mu_n\ast\sigma_n\ast \sum_{s< S(n-1)}b_s\|_1\\
&\leq&2\lambda^{-1}\sum_n\|\mu_n\|_1\sum_{s<S(n-1)} \|\sigma_n\ast b_s\|_1\\
&\leq&2\lambda^{-1}\sum_n \sum_{s<S(n-1)} 2^{-S(n-1)-n+s+1}\|b_s\|_1\\
&\leq& 2\lambda^{-1}\sum_n \sum_{s} 2^{-n+1}\|b_s\|_1\leq \frac C\lambda \|\varphi\|_1.
\end{eqnarray*}
(This is a standard Calder\'on-Zygmund argument so far.)  Now for the other sum, we can write
\begin{eqnarray*}
1-\hat\sigma_n(\gamma)=(1-e(\gamma))\displaystyle\sum_{j=0}^{2^{S(n-1)+n}-1} (1-j2^{-S(n-1)-n})e(j\gamma),
\end{eqnarray*}
and use (\ref{decaycdn}) to bound
\begin{eqnarray*}
\|\hat\mu_n(1-\hat\sigma_n)\|_\infty &\leq & 2^{S(n-1)+n}\sup_\gamma|(1-e(\gamma))\hat\mu_n(\gamma)| \leq 2^{-S(n-1)-n}.
\end{eqnarray*}
Here, in a variant of the technique from \cite{MC}, we will use the extremely strong $\ell^2$ estimate we get from this Fourier bound to obtain a weak $\ell^1$ estimate. Starting with Chebyshev's inequality,
\begin{eqnarray*}
|E_2| &\leq& 4\lambda^{-2}\left\|\sup_n \left|(\mu_n-\mu_n\ast\sigma_n)\ast \sum_{s< S(n-1)}b_s\right|\right\|_2^2\\
&\leq& 4\lambda^{-2} \sum_n \left\|(\mu_n-\mu_n\ast\sigma_n)\ast \sum_{s< S(n-1)}b_s\right\|_2^2\\
&=& 4\lambda^{-2} \sum_n \left\|\hat\mu_n(1-\hat\sigma_n) \sum_{s< S(n-1)}\hat b_s\right\|_2^2\\
&\leq& 4\lambda^{-2} \sum_n \|\hat\mu_n(1-\hat\sigma_n)\|_\infty^2\left\|\sum_{s< S(n-1)}\hat b_s\right\|_2^2\\
&\leq& 4\lambda^{-2} \sum_n 2^{-2S(n-1)-2n}\left\|\sum_{s< S(n-1)} b_s\right\|_2^2\\
&=& 4\lambda^{-2} \sum_n 2^{-2S(n-1)-2n} \sum_{s< S(n-1)}\|b_s\|_2^2
\end{eqnarray*}
using the orthogonality of $b_s$ for different $s$ for the last step.  Now since $\|b_s\|_\infty\leq\|b_s\|_1\leq\lambda 2^s$,
\begin{eqnarray*}
&\leq& 4\lambda^{-2} \sum_n 2^{-2S(n-1)-2n} \sum_{s< S(n-1)}\lambda 2^s\|b_s\|_1\\
&\leq& 4\lambda^{-1} \sum_n 2^{-S(n-1)-2n}\sum_s \|b_s\|_1\leq \frac C\lambda\|\varphi\|_1.
\end{eqnarray*}
This completes the proof.
\end{proof}

\begin{remark}
This proof generalizes straightforwardly to measure-preserving $\Z^d$-actions, and indeed, to actions by finitely generated abelian groups (this requires defining the Calder\'on-Zygmund decomposition on such a group, using for instance the dyadic cubes from \cite{MC2}).  Note that the proof of Theorem \ref{lpresult} for $p=2$ generalizes to this case, and thus we will have a.e. convergence of these ergodic averages for all $f\in L^1(X)$; see Theorem 2.4 in \cite{BJR}.
\end{remark}

\section{Negative Result: Rotated Averages on $n^2$}

In this section, we will prove Proposition \ref{rotsqrs}.  We start by adapting a definition from \cite{PETHA} that will be useful to us.
\begin{definition}
We say that a sequence of measures $\{\mu_n\}$ is \emph{persistently universally $L^p$-bad} if for every (non-atomic) ergodic dynamical system $(X,\F,\mu,\tau)$ and every infinite $\S\subset\N$, there exists an $f\in L^p(X,\mu)$ such that the sequence $\{\mu_n f(x)\}_{N\in\S}$ ($N$ taken in increasing order) diverges on a set of positive measure in $X$.
\end{definition}
Again, by the Conze principle \cite{Conze}, it can be shown that $\{\mu_n\}$ is persistently universally $L^p$-bad if and only if there is no subsequence $\{n_k\}$ such that the weak (p,p) maximal inequality
\begin{eqnarray*}
|\{x:\sup_k |\varphi\ast\mu_{n_k}(x)|>\lambda\}|\leq C\lambda^{-p}\|\varphi\|_{\ell^p(\Z)}^p\;\forall\varphi\in\ell^p(\Z)
\end{eqnarray*}
holds for some $C<\infty$.
\\ \\In \cite{LaVic2}, the averages along the squares $\nu_n:=\frac1n\sum_{k=1}^n \delta_{k^2}$ are proved to be persistently universally $L^1$-bad (they were proved universally $L^1$-bad earlier by Buczolich and Mauldin \cite{DAAS}).  While $\hat\mu_n$ does not converge to 0 everywhere, the behavior of these exponential sums is very regular; we will exploit this.\\
\\
Let us now consider the measures $\mu_n:=\displaystyle\frac1n\sum_{j=1}^n \delta_{j^2}e(n^{-1/2}j)$. \\ \\
Because $\mu_n\ast \varphi(k)=e(n^{-1/2}k)(\nu_n\ast (e(-n^{-1/2} \cdot)\varphi)(k))$, a weak maximal inequality that fails for $\{\nu_{n_k}\}$ will also fail for $\{\mu_{n_k}\}$, so that $\{\mu_n\}$ must be persistently universally $L^1$-bad.  It thus remains to show that $\hat \nu(\gamma)\to0$ for all $\gamma\in[0,1)$.
\\ \\
In the language of the Hardy-Littlewood circle method, this happens because each rational number never enters the ``major arc'' corresponding to itself, and because the peaks that pass over any point grow smaller as the denominator they correspond to grows larger.  But we shall be a little more rigorous than this.
\\ \\In this section and the next, we will use the classical result of Weyl \cite{Weyl} on trigonometric sums, and we will repeatedly refer to its exposition in Section II.2 of Rosenblatt and Wierdl \cite{PETHA} rather than replicate it in its entirety here\footnote{Many theorem and equation numbers in this paper will refer to \cite{PETHA}, rather than to Section 2 of this paper. Fortunately, none of the theorem or equation numbers will coincide between our Section 2 and their Section II.2.}.  In this section, we will require only the estimate from (2.24) of that paper:
\begin{lemma}
\label{weylbound}
Let $\beta\in[0,1)$ and $N\in\Z^+$, and let $p/q\in\Q$ with $(p,q)=1$, such that $q\leq N^{4/3}$ and $|\beta-p/q|\leq q^{-1}N^{-4/3}$.  Then
\begin{eqnarray}
\left| \frac1N \sum_{j=1}^N e(j^2\beta)\right|\leq C\left(\frac1{\sqrt q}+\frac{\sqrt {\log N}}{N^{1/3}}\right).
\end{eqnarray}
\end{lemma}
Note that for any $\beta$ and $N$ there exists such a $p/q$, by Dirichlet's theorem on rational approximations.  Now of course $\mu_n(\gamma)=\nu_n(\gamma+n^{-1/2})$.  Fix $\gamma\in[0,1)$, and for each $N$ let $q_N$ denote the smallest denominator such that for some integer $p_N$, $q_N\leq N^{4/3}$ and $|\gamma+N^{-1/2}-p_N/q_N|\leq q_N^{-1}N^{-4/3}$.
\\ \\ For each $a/b\in\Q$, if $\gamma\neq a/b$, then eventually $|\gamma -a/b|\geq N^{-1/2}+b^{-1}N^{-4/3}$ so that $q_N\neq b$ for $N$ sufficiently large. And if $\gamma=a/b$, then $|\gamma+N^{-1/2} -a/b|=N^{-1/2}\geq b^{-1}N^{-4/3}$, so that $q_N\neq b$.
\\ \\This implies that $q_N\to\infty$ as $N\to\infty$, and by Lemma \ref{weylbound}, we see that $\hat\nu(\gamma)\to0$ as $N\to\infty$.  This suffices to prove Proposition \ref{rotsqrs}.

\section{Threshold Result: Averages Along $n^2+\lfloor \rho(n)\rfloor$}
In this section, we will prove Theorem \ref{threshold}.
\\ \\We begin with the first claim, that if $\rho'(x)\gg x^{-1}$, then the $\{\mu_N\}$ have asymptotically trivial transforms.  As noted above, we will make extensive use of Section II.2 of \cite{PETHA} in this section.
\\ \\ Let $\beta\in\T$.  Again, by Dirichlet's theorem on rational approximations, there exists a rational number $p/q$ in lowest terms, with $q\leq N^{4/3}$, such that
\begin{eqnarray}
\label{dirichlet}
|\beta-p/q|\leq q^{-1}N^{-4/3}.
\end{eqnarray}
We first write
$$\hat\mu_N(\beta)= \frac1N\sum_{j=0}^{\lfloor \rho(N) \rfloor} e(j\beta)\sum_{k\in I_j}e(k^2\beta) + O\left(\frac{L_{\lfloor\rho(N)\rfloor}}N\right),$$
where $I_j:=\{x\in\R^+: \lfloor\rho(k)\rfloor=j\}$, and denote $L_j:=|I_j|$.  Note that by the hypotheses on $\rho$, we have $L_j$ increasing and $j^{\epsilon+2/3}\lesssim L_{\lfloor\rho(j)\rfloor} \ll j$ as $j\to\infty$.
\\ \\We adapt from \cite{PETHA} the notation
$$\hat\Lambda(p/q):= \frac1q\sum_{n=0}^{q-1} e(n^2p/q)$$
for $p/q\in\Q$, as well as the functions
$$V_j(\alpha):=\sum_{l:\sqrt l\in I_j}\frac1{2\sqrt l}e(l\alpha)$$
for any $\alpha\in\R$.
\\ \\As with the estimate (2.24) there, we will make use of different estimates depending on the size of the denominator $q$ relative to $N$.
\begin{lemma}
There exists a constant $C<\infty$, depending only on $\rho$ and $\epsilon$, such that for $j>\rho(N^{1-\epsilon})$, $q\leq N^{2/3}$, and $|\beta-p/q|\leq q^{-1}N^{-4/3}$, we have
\begin{eqnarray}
\label{qsmall}
\left| \sum_{k\in I_j}e(k^2\beta) - \hat\Lambda(p/q)V_j(\beta-p/q)\right| \leq CN^{-\epsilon/6}L_j.
\end{eqnarray}
Similarly, if  $j>\rho(N^{1-\epsilon})$, $N^{2/3}<q\leq N^{4/3}$, and $|\beta-p/q|\leq q^{-1}N^{-4/3}$, then
\begin{eqnarray}
\label{qlarge}
\left| \sum_{k\in I_j}e(k^2\beta)\right| \leq CN^{-\epsilon/7}L_j.
\end{eqnarray}
\end{lemma}
\begin{proof}
(\ref{qsmall}) is the equivalent of (2.25), replacing the exponential sums beginning at 0 with the sums along the interval $\{l:\sqrt l\in I_j\}$. As in (2.27),
\begin{eqnarray*}
\left| \sum_{k\in I_j}e(k^2p/q) - \hat\Lambda(p/q)V_j(0)\right| \leq Cq
\end{eqnarray*}
clearly holds with a universal constant.  Thus we may apply Lemma 2.13 (note that we are summing in $l$, over $\leq 2NL_j$ terms).  Thus we find
\begin{eqnarray*}
\left| \sum_{k\in I_j}e(k^2\beta) - \hat\Lambda(p/q)V_j(\beta-p/q)\right| \leq Cq(NL_j|\beta-p/q|+1)\leq CN^{-\epsilon/6}L_j,
\end{eqnarray*}
using for the $Cq$ term the assumption $q\leq N^{2/3}$ and the fact that for $j>\rho(N^{1-\epsilon})$, we have $$L_j\gtrsim N^{(1-\epsilon)(2/3+\epsilon)}\gtrsim N^{2/3+\epsilon/6}.$$
Similarly, (\ref{qlarge}) is the analogue of (2.28), and the required estimate
\begin{eqnarray*}
\left| \sum_{k\in I_j}e(k^2p/q)\right| \leq C(\frac{L_j}{\sqrt q} + \sqrt{q \log L_j}),
\end{eqnarray*}
like (2.29), relies only on the fact that the squares are summed along an interval. We therefore apply Lemma 2.13 to obtain
\begin{eqnarray*}
\left| \sum_{k\in I_j}e(k^2\beta)\right| \leq C(\frac{L_j}{\sqrt q} + \sqrt{q \log L_j})(NL_j|\beta-p/q|+1),
\end{eqnarray*}
which for $q>N^{2/3}$ is indeed bounded by $CN^{-\epsilon/7}L_j$.
\end{proof}
Thus we have a satisfactory bound for $q>N^{2/3}$, while for $q\leq N^{2/3}$ and $|\beta-p/q|\leq q^{-1}N^{-4/3}$, we have
\begin{eqnarray*}
\hat\mu_N(\beta)= \frac1N\sum_{j=0}^{\lfloor \rho(N) \rfloor} e(j\beta)\hat\Lambda(p/q)V_j(\beta-p/q) + O\left(\frac{L_j}N+N^{-\epsilon/6}\right).
\end{eqnarray*}
Now 
\begin{eqnarray*}
\frac1N\sum_{j=0}^{\lfloor \rho(N) \rfloor} e(j\beta)V_j(\beta-p/q)
\end{eqnarray*} is a Cesaro mean of the averages
\begin{eqnarray}
\label{expos}
\frac1{N^2}\sum_{j=0}^{\lfloor \rho(N) \rfloor} e(j\beta)\sum_{l:\sqrt l\in I_j}e(l(\beta-p/q)),
\end{eqnarray}
 and so it suffices for Theorem \ref{threshold} to show that these averages decay to 0 (away from $\beta=0$) at a rate depending only on $\rho$.  Let $\alpha=\beta- p/q$.  For $\alpha\neq0$,
\begin{eqnarray*}
\frac1{N^2}\sum_{j=0}^{\lfloor \rho(N) \rfloor} e(j\beta)\sum_{l:\sqrt l\in I_j}e(l\alpha)&=& \frac1{N^2}\sum_{j=0}^{\lfloor \rho(N) \rfloor}e(j\beta)\frac{e(\lceil\phi(j+1) \rceil\alpha)-e(\lceil\phi(j) \rceil\alpha)}{e(\alpha)-1}
\end{eqnarray*}
where $\phi(j):= (\rho^{-1}(j))^2$.  We use the sum version of integration by parts, i.e. $\sum_{j=0}^m (a_{j+1}-a_j)b_j=a_{m+1}b_m-a_1b_0+\sum_{j=0}^m a_j(b_{j-1}-b_j)$, with $a_j:=N^{-2}\sum_{i=0}^{j-1} e(i\beta)$ and $b_j:=\frac{e(\lceil\phi(j+1) \rceil\alpha)-e(\lceil\phi(j) \rceil\alpha)}{e(\alpha)-1}$.  We evaluate the end terms first, noting that
\begin{eqnarray*}
|a_{j+1}b_j|\leq\frac2{N^2|\beta|}\left|\frac{e(\lceil\phi(j+1) \rceil\alpha)-e(\lceil\phi(j) \rceil\alpha)}{e(\alpha)-1}\right|\leq\frac{C(\phi(j+1)-\phi(j))}{N^2|\beta|}=O\left(\frac{L_j}{N|\beta|}\right).
\end{eqnarray*}
Now the main sum is
\begin{eqnarray*}
\left|\sum_{j=0}^{\lfloor \rho(N) \rfloor} a_j(b_{j-1}-b_j)\right|&=&\left| \frac1{N^2}\sum_{j=0}^{\lfloor \rho(N) \rfloor} \frac{1-e(j\beta)}{1-e(\beta)}\frac{e(\lceil\phi(j+1) \rceil\alpha)-2e(\lceil\phi(j) \rceil\alpha)+e(\lceil\phi(j-1) \rceil\alpha)}{1-e(\alpha)}\right|\\
&\leq&\frac2{N^2|\beta||\alpha|}\left| \sum_{j=0}^{\lfloor \rho(N) \rfloor}e(\lceil\phi(j+1) \rceil\alpha)-2e(\lceil\phi(j) \rceil\alpha)+e(\lceil\phi(j-1) \rceil\alpha)\right|\\
&\leq& \frac C{N^2|\beta||\alpha|}\sum_{j=0}^{\lfloor \rho(N) \rfloor} \phi''(j+1)|\alpha|\leq \frac{C\phi'(N)}{N^2|\beta|}=O\left(\frac{L_{\lfloor \rho(N) \rfloor}}{N|\beta|}\right),
\end{eqnarray*}
using the monotonicity of $\rho$ and its derivatives (and thus the monotonicity of $\phi''$) to justify the second inequality.  \\ \\
If $\alpha=0$, then by the same methods,
\begin{eqnarray*}
\frac1{N^2}\sum_{j=0}^{\lfloor \rho(N) \rfloor} e(j\beta)|\{l:\sqrt l\in I_j\}|=\frac1{N^2}\sum_{j=0}^{\lfloor \rho(N) \rfloor} e(j\beta)(\lceil\phi(j+1) \rceil-\lceil\phi(j) \rceil)=O\left(\frac{L_{\lfloor \rho(N) \rfloor}}{N|\beta|}\right).
\end{eqnarray*}
Therefore we have proved that
\begin{eqnarray*}
|\hat\mu_N(\beta)| \lesssim N^{-\epsilon/7}+ \frac{L_{\lfloor \rho(N) \rfloor}}{N|\beta|},
\end{eqnarray*}
which clearly establishes that the sequence $\{\mu_N\}$ has asymptotically trivial transforms.
\\ \\
Now we turn to the second claim, that if $\rho'(x)\leq Cx^{-1}$ (thus $\rho(x)\leq C\log x$), then $k^2+\lfloor \rho(k)\rfloor$ is persistently universally $L^1$-bad.  This follows from the argument of \cite{LaVic2}.  Given such a $\rho$ and any subsequence $\{N_k\}$, we choose a further subsequence $\{k_i\}$ such that modulo any squarefree odd $Q$, $\lfloor\rho(\frac12N_{k_i})\rfloor$ has a limit $r_Q$ as $i\to\infty$. (This is done by a diagonal argument, since we only need to ensure that this happens modulo the product of the first $M$ primes, for each $M$.)  Then if we restrict to this subsequence, we see that the quadratic residues translated by $r_q$ serve as the $\Lambda_q$ in Theorem 3.1 of that paper, that (3.1)-(3.4) hold for the same reasons as for the original quadratic residues, and that for any nonzero quadratic residue $a$ modulo a squarefree and odd $Q$ with sufficiently large factors,
\begin{eqnarray*}
& &\liminf_ {k\to\infty} \frac1{N_k}|\{1\leq j\leq N_k: j^2+\lfloor\rho(j)\rfloor \equiv a+r_Q \mod Q\}|\\
&\geq& \liminf_ {k\to\infty} \frac1{N_k}|\{1\leq j\leq N_k: j^2\equiv a \mod Q, \lfloor\rho(j)\rfloor = \lfloor\rho(N_k/2)\rfloor\}|\\
&\geq& \liminf_ {k\to\infty}|\Lambda_Q|^{-1}\frac1{2N_k}|\{1\leq j\leq N_k: \lfloor\rho(j)\rfloor = \lfloor\rho(N_k/2)\rfloor\}| -\frac1Q\\
&\geq& \frac1{3C|\Lambda_Q|}
\end{eqnarray*}
using the fact that $\rho'(x)\leq \frac{4C}N\,\forall x\geq N/4$, so that $|\{1\leq j\leq N_k: \lfloor\rho(j)\rfloor = \lfloor\rho(N_k/2)\rfloor\}|\geq\min\{\frac{N}{4C},\frac{N}4\}$, and the fact that $|\Lambda_Q|\ll Q$ for $Q$ large.  Note that Theorem 4.1 in \cite{LaVic2} implies this variant of Theorem 3.1, since (4.7) is the only use of (3.5) in that paper.  Therefore any subsequence of the $\mu_N$ has a further subsequence which is universally $L^1$-bad, which implies our desired result.

\subsection*{Acknowledgements}
The author thanks J. Rosenblatt for introducing him to these interesting problems, for pointing out the article \cite{MR2177432}, and for many suggestions; and his dissertation advisor, M. Christ, for several major ideas, including treating the intermediate terms in Theorem \ref{weakmax} separately after the fashion of \cite{STW}.
\\ \\This research was partly supported by the NSF (Grant DMS-0401260.).

\bibliography{BibLaVic03}{}
\bibliographystyle{plain}

\end{document}